\documentclass{amsproc}

\title{Virtual classes for hypersurfaces via two-periodic complexes}

\author{Mark Shoemaker}
\address{
Department of Mathematics, Colorado State University, 1874 Campus Delivery, Fort Collins, CO, USA, 80523-1874
}
\email{mark.shoemaker@colostate.edu} 
\thanks{This work was partially supported by NSF grant DMS-1708104.}

\usepackage[mathscr]{eucal}
\usepackage{palatino, mathpazo, amsfonts, mathrsfs, amscd}
\usepackage[all]{xy}
\usepackage{hyperref}
\usepackage{url}
\usepackage{amssymb, amsmath, amsthm}
\usepackage{stackrel}
\usepackage{bm}
\usepackage{tikz-cd}
\usepackage{tikz}
\usetikzlibrary{arrows}
\usepackage{pdfpages}

\newtheorem{theorem}{Theorem}[section]
\newtheorem{lemma}[theorem]{Lemma}

\newtheorem{proposition}[theorem]{Proposition}
\newtheorem{corollary}[theorem]{Corollary}

\theoremstyle{remark}
\newtheorem{remark}[theorem]{Remark}
\newtheorem{example}[theorem]{Example}
\newtheorem{definition}[theorem]{Definition}

\newtheorem{assumption}[theorem]{Assumption}

\usepackage{todonotes}

\makeatletter
\def\imod#1{\allowbreak\mkern10mu({\operator@font mod}\,\,#1)}
\makeatother

\newcommand{\cc}[1]{\mathcal{#1}}  

\newcommand{\CC}{\mathbb{C}}
\newcommand{\ZZ}{\mathbb{Z}}
\newcommand{\LL}{\mathbb{L}}

\newcommand{\PP}{\mathbb{P}}

\newcommand{\RR}{\mathbb{R}}

\newcommand{\fC}{\mathfrak{C}}

\newcommand{\Bun}{\mathfrak{Bun}}

\newcommand{\sMbar}{\overline{\mathscr{M}}}

\newcommand{\sM}{{\mathscr M}}

\newcommand{\cC}{\cc C}
\newcommand{\cL}{\cc L}
\newcommand{\cN}{\cc N}
\newcommand{\cM}{\cc M}
\newcommand{\cO}{\cc O}
\newcommand{\cV}{\cc V}
\newcommand{\cQ}{\cc Q}
\newcommand{\cA}{\cc A}
\newcommand{\cB}{\cc B}

\newcommand{\uC}{\underline{\cc C}}

\DeclareMathOperator{\st}{st}

\DeclareMathOperator{\vir}{vir}
\DeclareMathOperator{\taut}{taut}

\DeclareMathOperator{\ch}{ch}
\DeclareMathOperator{\td}{Td}

\DeclareMathOperator{\tot}{tot}

\DeclareMathOperator{\rank}{rank}

\DeclareMathOperator{\Sym}{Sym}

\DeclareMathOperator{\Spec}{Spec}

\newcommand{\op}[1]{\operatorname{#1}}

\numberwithin{equation}{section}

\begin{document}

\begin{abstract}
These expository notes are based on a series of lectures given at the May 2018 Snowbird workshop, Crossing the Walls in Enumerative Geometry.  We give an introductory treatment of the notion of a virtual fundamental class in algebraic geometry, and describe a new construction of the virtual fundamental class for Gromov--Witten theory of a hypersurface.  The results presented here are based on joint work with I.~Ciocan-Fontanine, D.~Favero, J.~Gu\'er\'e, and B.~Kim.

\end{abstract}

\maketitle

\tableofcontents

\section*{Introduction}
\subsection{Setup}

Fix $X$ to be a smooth projective variety.  The moduli space
$\sMbar_{g,n}(X,d)$ parametrizes maps \[f\colon C \to X,\] where $C$ is a smooth or nodal curve of genus $g$ with $n$ marked points,  $f$ is a regular map of degree $d$, and the automorphisms of $C$ which preserve $f$ are required to be finite.  One considers integrals over these moduli spaces, with the hopes of obtaining useful invariants of the variety $X$.  

A special class of such integrals yield what are known as \emph{Gromov--Witten invariants} of $X$.  In particularly simple cases these invariants give counts of curves in $X$, although in general such a direct enumerative interpretation is not possible.  As predicted by physicists, these integrals should satisfy a number of remarkable properties.  For instance, they should be invariant under deformations of $X$, and should be compatible with the following natural maps between the moduli spaces curves 
\begin{align}\label{e:gl}
\op{for}\colon &\sMbar_{g,n+1} \to \sMbar_{g,n}, \\ \nonumber
\op{gl}_1 \colon &\sMbar_{g,n+2} \to \sMbar_{g+1,n}, \\
\op{gl}_2  \colon& \sMbar_{g_1, n_1 + 1} \times \sMbar_{g_2,n_2+1} \to \sMbar_{g_1 + g_2,n_1 + n_2}.
\nonumber
\end{align}
in a precise sense.  


In attempting to rigorously define Gromov--Witten invariants, a serious technical challenge arises.  In contrast to $\sMbar_{g,n}$, the moduli space of stable maps $\sMbar_{g,n}(X,d)$ is generally not smooth and often contains many irreducible components of different dimension.  
The fundamental class of $\sMbar_{g,n}(X,d)$ is not well-defined;
it is therefore unclear how one should ``integrate'' over $\sMbar_{g,n}(X,d)$.  



\subsection{The virtual fundamental class}

The solution is the notion of a \emph{virtual fundamental class}, a choice of element in the Chow group of $\sMbar_{g,n}(X,d)$ which behaves in many ways like a fundamental class.  
To be more precise, one hopes to construct an element of $A_*\left( \sMbar_{g,n}(X,d)\right)$ which satisfies a number of properties analogous to those of a fundamental class, together with additional properties reflecting the goals of the previous section.  In particular one requires:
\begin{enumerate}
\item the virtual fundamental class is supported entirely in a single degree of $A_*\left( \sMbar_{g,n}(X,d)\right)$, the so-called \emph{expected dimension} (see Section~\ref{ss:vcp});
\item the virtual fundamental class is compatible with the various gluing and forgetful maps of \eqref{e:gl} (see \cite[Section~2]{Beh2} for precise statements);
\item Gromov--Witten invariants of $X$, defined by integrating certain classes against the virtual fundamental class, are invariant under deformations of $X$.
\end{enumerate}
It is not clear a-priori that such a class should exist, and in general there is no guarantee that a choice of such a class is unique.  The construction of the virtual fundamental class has been a fundamental challenge in Gromov--Witten theory (as well as other curve counting theories).
Unfortunately any construction of the virtual fundamental class is necessarily technical, especially in its full generality.  As such, the virtual fundamental class is one of the more formidable hurdles for graduate students learning the theory.

Constructing the virtual fundamental class has now been successfully carried out in a number of different ways.  In Gromov--Witten theory this has been done by 
Behrend--Fantechi in \cite{BF}, Fukaya--Ono  in \cite{FO}, Li--Tian in \cite{LiT},
Ruan in \cite{RuanV}, 
 and others.

\subsection{Summary}

These expository notes have two complementary goals.  The first is to give an introduction to the concept of a virtual fundamental class.  This is done in Section~\ref{ss:zerolocus}, where we explain the connection to more classical intersection theory.  In Section~\ref{s:proj} we detail the  case of $\sMbar_{g,n}(\PP^r,d)$ where the construction can be  simplified.  

The second goal of these notes is to present a new method, developed by
Ciocan-Fontanine,
Favero,
Gu\'er\'e, 
Kim, and the author in \cite{CFFGKS},
 of constructing a virtual class for the Gromov--Witten theory of a hypersurface using two-periodic complexes of vector bundles.  This is done in Sections~\ref{s:ne0} and~\ref{s:ts}.
   

%

\subsection{Audience} \label{ss:aud}

These notes are intended for an audience who has some familiarity with Gromov--Witten theory, 
 but has perhaps not yet learned the technical details of the virtual fundamental class.
 While this may seem to be a rather specific collection of readers, in our experience it includes a large proportion of graduate students learning Gromov--Witten theory.
We assume readers have a familiarity with algebraic geometry at the level of at least the first three chapters of Hartshorne \cite{Hart}.  A sufficient pre-requisite for the Gromov--Witten theory material is, e.g., the excellent introductory text by Kock and Vainsencher \cite{KV}. 

On the other hand, while Sections~\ref{ss:zerolocus} and~\ref{s:proj} are certainly well-known to many, we hope that the later sections will be of interest to experts in the field, as they provide a new construction of the virtual fundamental class on the space of maps to a hypersurface.

\subsection{A note on generality}
The method described in these notes is a particular case of a more general construction developed 
in \cite{CFFGKS}.  The  setting in which these techniques apply is that of a 
 \emph{gauged linear sigma model} (GLSM), which simultaneously generalizes Gromov--Witten theory of complete intersections as well as FJRW theory \cite{FJR13, PV} of homogeneous singularities. GLSMs were first described in physics by Witten in \cite{Wi2}, while the first mathematical definition was given
by Fan--Jarvis--Ruan in \cite{FJR15}.  

In the particular case of FJRW theory, Polishchuk--Vaintrob used two-periodic complexes (and more generally the derived category of factorizations) to give an algebraic construction of a virtual class in \cite{PVold, PV}.  
The paper \cite{CFFGKS} was inspired by the work in  \cite{FJR15} and \cite{PV}.
In these notes we use two-periodic complexes in a manner analogous to that of \cite{PVold} to define Gromov--Witten-like invariants for hypersurfaces.  For a more detailed discussion of connections to other results, see Section~\ref{s:conc}.

Although we do not make further mention of the GLSM, it is in the background of everything that follows.
 We hope that by focusing on a special case, these notes will provide a simplified roadmap for readers interested in \cite{CFFGKS}.

\subsection{Acknowledgments}  These notes are based on a series of lectures given at the 2018 Crossing the Walls in Enumerative Geometry Summer Workshop, and those lectures are based upon the paper \cite{CFFGKS}.  I thank the organizers of the workshop,
T.~Jarvis,
N.~Priddis, and
Y.~Ruan, as well as
D. Coffman,
the project manager who took care of much of the work required to make the workshop such a success.  
I am grateful to my collaborators on \cite{CFFGKS}, 
I.~Ciocan-Fontanine,
D.~Favero,
J.~Gu\'er\'e, and
B.~Kim.
I thank the University of Minnesota's Department of Mathematics for the use of their facilities during several extended collaborative meetings during the writing of \cite{CFFGKS}.
I am grateful to T.~Jarvis and Y.~Ruan for many discussions explaining their beautiful construction of the GLSM.
Finally, I am grateful to the referee for providing feedback and suggesting many improvements to the first draft of this manuscript.

%



\section{The zero locus of a section}\label{ss:zerolocus}
In this section we describe a relatively simple setting in which a virtual fundamental class arises, the zero locus of a section of a vector bundle.  This is referred to as the \emph{basic model} in \cite[Appendix B]{PandTho}.
In this setting, a virtual fundamental class can be defined using the techniques of intersection theory and is given by a more familiar object, the refined Euler class.
We will see in the next  section that the virtual class for $\sMbar_{g,n}(\PP^r, d)$ is in fact a particular case of this example.  The material of this section is described in detail in \cite[Chapter~6]{Fu}.

Let $X$ be a smooth variety of dimension $n$ and let $E \to X$ be a vector bundle of rank $r$.  Let $s \in \Gamma(X, E)$ be a section, and define $Z := Z(s)$ to be the scheme-theoretic zero locus of the section $s$.  In this setting, $Z$ has an \emph{expected dimension}, given by $n - r$.  If $s$ is a regular section of $E$, then this is the dimension of $Z$.  

We hope to construct a class $[Z]^{\vir}$ in $A_*(Z)$, called the virtual fundamental class, satisfying the following properties:
\begin{enumerate}
\item the class $[Z]^{\vir}$ is supported in the expected dimension: $[Z]^{\vir} \in A_{n-r}(Z)$;
\item \label{defprop} the class $[Z]^{\vir}$ is invariant under deformations in the following sense: if $s'$ is another section of $E$ such that $Z' = Z(s') \subset Z$, then $[Z]^{\vir}$ and $[Z ']^{\vir}$ are equal in $A_*(Z)$;
\item if $s(X)$ intersects the zero section of $E$ transversally, then $[Z]^{\vir} = [Z]$.
\end{enumerate}
Such a class gives an invariant of  $Z$ which remembers the fact that $Z$ arose from a section of $E$.  If $s$ is not transverse to the zero section, one may attempt to construct such a class by deforming $s$ to some $s'$ which is transverse, and such that $Z' = Z(s') \subset Z$.  If such a deformation is possible, one could define the virtual fundamental class of $Z$ to be $[Z '] \in A_*(Z)$.  In many settings, however, such a section does not exist.  One motivation for the deformation property \eqref{defprop} above is computation.  Sometimes one can deform the section $s$ to a section $s'$ for which the virtual class is easier to compute.



If $s \in \Gamma(X, E)$ is not a regular section, then $\dim(Z)$ will be strictly greater than the expected dimension $n-r$.  We explain below how to construct a class in the Chow group of $Z$ which 
satisfies the properties given above.

\begin{example}
Let us consider what is in some sense the worst possible case, that is, $s \in \Gamma(X, E)$ is the zero section $0 \in \Gamma(X, E)$. 
In this case we define
\[[Z]^{\vir} = e(E) \in A_{n-r}(X),\]
where $e(E)$ denotes the \emph{Euler class} (the top Chern class) of $E$ \cite[Chapter 3]{Fu}.  If $X$ is compact, then under the cycle map $A_*(X) \to H_*(X) \cong H^*(X)$, $e(E)$ maps to the topological Euler class of, e.g. \cite{BT}.

If there exists a section $s' \in \Gamma(X, E)$ transverse to the zero section, and  $Z' = Z(s ')$ is the zero locus, then the class $[Z '] \in A_{n-r}(X)$ is equal to $e(E)$.  In this sense $e(E)$ is the correct replacement for $[X]$, as it coincides with the class obtained by deforming $s$ into $s'$ and taking the fundamental class of the deformed zero locus.
\end{example}

For a more general section $s \in \Gamma(X, E)$, not necessarily regular or zero, there is a notion of a \emph{refined Euler class} 
$e(E, s)$ in $A_{n-r}(Z)$.
It agrees with the usual Euler class in the sense that under the pushforward map
\[i_*\colon A_*(Z) \to A_*(X),\]
$e(E, s)$ maps to $e(E)$.  

\begin{definition}\label{conelim}
Given $Z = Z(s)$ as above, the \emph{normal cone} $C_Z X$ is defined to be the limit of the graph of the section $\lambda \cdot s$ as $\lambda \mapsto \infty$.  More precisely, consider the map 
\begin{align*}
X \times \CC &\to E \times \CC^*\\
(x, \lambda) & \mapsto \lambda s(x).
\end{align*}
Let $\Gamma$ denote the image of this map and let $\overline \Gamma$ denote its closure in $E \times \PP^1$.  Define $C_Z X$ to be the intersection $\tot (E) \times \{\infty\}$ with $\overline \Gamma$.
\end{definition}
The scheme $\overline \Gamma \to \CC$ deforms the subvariety $s(X) \subset E$ to $C_Z X$.  If $s(x) \neq 0$ for a given point $x$, $\lambda s(x)$ will go to infinity as $\lambda$ goes to $\infty$.  As a consequence, $C_Z X$ is contained entirely in $E|_Z$.  In \cite[Appendix B.6.6]{Fu} it is shown that $C_Z X$ is of pure dimension $n$.



The projection $\pi|_Z\colon E|_Z \to Z$ is flat, and so defines a map \begin{align*}\pi|_Z^*\colon A_k(Z) \to & A_{k + r}(E|_Z) \\ [V] \mapsto & [\pi|_Z^{-1}(V)].\end{align*}  It is a (nontrivial) fact that this map is an isomorphism.  We denote its inverse by $0|_Z^*\colon A_{k + r}(E|_Z) \to A_k(Z)$, as we may view it as pullback by the zero section  $0|_Z\colon Z \to E|_Z$. 
With this setup, one can define the following:
\begin{definition}\label{d:eref} The refined Euler class of $E \to X$ with respect to $s \in \Gamma(X, E)$ is the pullback of the normal cone of $Z$ in $X$ via the zero section:
\[e(E, s) := 0|_Z^*\left([C_{Z}X]\right).\]
\end{definition}

The pullback preserves codimension, so this class lies in $A_{n-r}(Z)$.
\begin{definition}\label{d:refeu} Define the virtual class 
\[[Z]^{vir} := e(E, s).\]
\end{definition}
It is an exercise to check that this class satisfies the properties desired of a virtual fundamental class, and agrees with the previous definition in the special case where $s$ is transverse to the zero section or $s = 0$.

We will see below that the virtual class for the moduli space of stable maps to projective space is in fact a special case of this construction.


\section{Stable maps to projective space}\label{s:proj}
In this section we construct a virtual fundamental class on $ \sMbar_{g,n}(\PP^r, d)$ using the refined Euler class.  
\subsection{Spaces of sections}\label{ss:sos}
We begin by defining a virtual fundamental class for a space of sections of a vector bundle over a family of curves.  We define the moduli space of sections below and realize it as the zero locus of a section of a vector bundle.  In Section~\ref{ss:smssos} we will see that $ \sMbar_{g,n}(\PP^r, d)$ itself is a space of sections.  

Let $\pi\colon \cC \to S$ be a flat (and proper) family of pre-stable curves lying over a smooth variety (or stack) $S$.  Let $\cV \to \cC$ be a vector bundle on $\cC$.  For each geometric point $s \in S$, the fiber of $\pi$ is a (at worst) nodal curve $C_s$, equipped with a vector bundle $V_s \to C_s$. 
\[
\begin{tikzcd}
 V_s \ar[d] \ar[r, hookrightarrow] & \cV \ar[d] \\
C_s \ar[d, "\pi_s"] \ar[r, hookrightarrow] &  \cC \ar[d, "\pi"]  \\
s \ar[r, hookrightarrow] &  S. 
\end{tikzcd}
\]

\begin{definition}\label{d:sections}
Let $\pi\colon \cC \to S$ and $\cV \to \cC$ be defined as in the introduction to this section.  Define the \emph{space of sections} of $\cV$ to be
\[ \tot(\pi_*\cV) := \Spec \left( \Sym \left( \RR^1 \pi_*(\omega_\pi \otimes \cV^\vee)\right)\right),\]
where $\omega_\pi$ is the relative dualizing sheaf for $\pi$.
\end{definition}
Given a sheaf $\cc A$ over $S$, recall that the $s$-points of $\Spec\left( \Sym(\cc A)\right)$ are the elements of $\hom(\cc A, \cO_s)$.  With this we observe that 

\begin{align} \label{a:spoints}
\tot(\pi_*\cV)|_s &= \hom \left(\RR^1 \pi_*(\omega_\pi \otimes \cV^\vee)_s, \cO_s \right) \\
& = \hom \left(\RR^1 {\pi_s}_*(\omega_{\pi_s} \otimes V_s^\vee), \cO_s \right) \nonumber \\
& = \Gamma(C_s, V_s), \nonumber
\end{align}
justifying the terminology space of sections.
Here the second equality is by base change (see \cite{EGAII}, or \cite[Theorem 1.2]{Oss} for a summary of the case at hand). The third equality is Serre duality.  In fact a stronger statement is true.
\begin{definition}\cite[Definition~2.1]{CL} 
Define the contravariant functor 
\[\op{sec}_{\cV}\colon \op{Schemes} \to \op{Set}\] which sends a scheme $T$ to the set of pairs $(\rho, p)$, where $\rho\colon T \to S$ is a morphism and $p  \in H^0(\cC_T, \rho^*\cL)$ where $\cC_T = T \times_S \cC $ and $\rho^*\cL = \cC_T \times_\cC \cL$.  A morphism between $(p, \rho)$ and $(p', \rho ')$  is given by a morphism $\tau\colon \cC_T \to \cC_T$  such that $\rho ' \circ \tau  = \rho$ and under the induced isomorphism $\rho^*\cL \to \tau^*{\rho '}^* \cL$, $p = \tau^*(p')$.
\end{definition}
By a similar argument to \eqref{a:spoints} above, one obtains the following.
\begin{proposition}\cite[Proposition~2.2]{CL}\label{p:dirimcone}
The stack $ \tot(\pi_*\cV)$ represents the functor $\op{sec}_{\cV}$.  \end{proposition}

\begin{remark}
The notation $\tot(\pi_*\cV)$ is an abuse of notation as the space of sections is not in general a vector bundle and $\pi_*\cV$ is not locally free.  We hope this will not cause confusion. 
%
\end{remark}

The space $\tot(\pi_* \cV)$ is not usually smooth and in fact is not even of pure dimension.  Nevertheless we can exploit the ideas of Section~\ref{ss:zerolocus} to construct a Chow class on $\tot(\pi_* \cV)$ of relative dimension $\chi(C_s, V_s)$ over $S$.  
We will use this in the next section to define a virtual class for $\sMbar_{g,n}(\PP^r,d)$.

First, we embed $\cV$ into a $\pi$-acyclic vector bundle $\cA \to \cC$.
\begin{lemma}
If $\pi\colon \cC \to S$ is projective, there exists an embedding of $\cV$ into a vector bundle $\cA$ satisfying $\RR^1 \pi_*(\cA) = 0$.
\end{lemma}

\begin{proof}
Let $\cO(1) \to \cC$ be a $\pi$-relatively ample line bundle.  Then for sufficiently large $n$, 
\[\RR^1 \pi_*(\cV^\vee(n)) = \RR^1 \pi_*(\cO(n)) = 0\] and the map 
\[ \pi^*(\pi_*(\cV^\vee(n))) \to \cV^\vee(n)\]
is surjective.  Twisting by $\cO(-n)$, we obtain the surjective map
\[ \pi^*(\pi_*(\cV^\vee(n)))(-n) \to \cV^\vee.\]
  Dualizing this map gives an embedding $\cV$ into 
\[\cA := \pi^*(\pi_*(\cV^\vee(n))^\vee)(n).\]
\end{proof} 
Given such an embedding $\cV \hookrightarrow \cA$, let $\cB$ denote the cokernel.  Define 
\begin{align*}A :=& \pi_*(\cA), \\
B :=& \pi_*(\cB).\end{align*}
The short exact sequence
\[ 0 \to \cV \to \cA \to \cB \to 0\]
induces a long exact sequence
\begin{equation}\label{e:LES1}
\begin{tikzcd}[row sep= tiny, column sep= small]
0 \ar[r] &\pi_*\cV \ar[r] & A  \ar[r]& B &  \\
\ar[r]& \RR^1\pi_*\cV   \ar[r]& \RR^1\pi_*\cA  \ar[r]& \RR^1\pi_*\cB  \ar[r]& 0.
\end{tikzcd}
\end{equation}
Note, however, that $\RR^1\pi_*\cA = 0$ by construction and so $\RR^1\pi_*\cB = 0$ as well.  Because the Euler characteristic is constant in flat families, the fibers of $A$ and $B$ have constant rank.  By Grauert's criterion \cite[12.9]{Hart}, they are vector bundles.
We conclude that the two-term complex $[A \to B]$ is a resolution of $\RR \pi_*\cV$ by vector bundles.  

Consider the total space of $A$, 
\[X:= \tot(A),\] and the map $\tau\colon X \to S$ which forgets the section of $\cA$.  Define $E := \tau^*(B)$ over $X$.  The map $A \to B$ induces a section $s \in \Gamma(X, E)$.  By \eqref{e:LES1}  the zero section $Z(s) \hookrightarrow X$ is exactly $\tot(\pi_*\cV)$.  We are now in the situation of Section~\ref{ss:zerolocus}.
%
%
We proceed as before.
\begin{definition}
Define the virtual fundamental class of the space of sections $\tot(\pi_*\cV)$ to be
\[ [\tot(\pi_*\cV)]^{\vir} := e(E, s) \in A_*(\tot(\pi_*\cV)).\]
\end{definition}
Note in the above that $[\tot(\pi_*\cV)]^{\vir}  \in A_k(\tot(\pi_*\cV))$, where 
\begin{align*} k = &\dim(\tot(A)) - \rank(E) \\
=& \dim(S) + \rank(A) - \rank(B) \\ = &\dim(S) + \chi(C_s, V_s),\end{align*}
for $s$ a geometric point of $S$.

One can check that the virtual class is well-defined in the following sense.
\begin{proposition}
The class $[\tot(\pi_*\cV)]^{\vir}$ is independent of the choice of embedding $\cV \to \cA$.
\end{proposition}
\begin{proof}
This follows from more general constructions such as in \cite{BF}.  However in this simple case one can prove it more easily.  The sketch of the proof is the following:
\begin{enumerate}
\item Reduce to the case of two embeddings $\cV \to \cA$ and $\cV \to \cA '$, where the second map factors through an embedding $\cA \to \cA '$.  This can be accomplished by embedding $\cA$ and $\cA '$ in a common larger $\pi$-acyclic vector bundle $\widetilde{\cA}$.
\item By a deformation argument, reduce to the case that $\cA ' = \cA \oplus \cA ''$ and the embedding $\cA \to \cA '$ is given by $(\op{id}_{\cA}, 0)$ (see \cite[Proposition~18.1]{Fu} for such an argument).
\item In this simple case, one observes from the definition that $e(E', s')$ and $e(E, s)$ coincide.
\end{enumerate}
\end{proof}

\subsection{Stable maps as a space of sections}\label{ss:smssos}
In this section we realize $\sMbar_{g, n}(\PP^r, d)$ as (an open subset of) a space of sections and use the results of Section~\ref{d:sections} to construct the virtual fundamental class.  We conclude by showing that this construction agrees with the definition of Behrend and Fantechi in terms of the intrinsic normal cone \cite{BF, Beh}.

Recall that the moduli space $\sMbar_{g, n}(\PP^r, d)$ is the stack representing families of genus $g$ curves with $n$ marked points, together with a degree $d$ map from the curve to $\cC$.  More precisely, a family of (genus $g$, $n$-marked, degree $d$) stable maps over $S$ is given by a diagram:
\[
\begin{tikzcd}
\cC \ar[d, "\pi"] \ar[r, "f"] & \PP^r \\
S \ar[u, bend left, "\{\sigma_i\}_{i=1}^n"] &
\end{tikzcd}
\]
where:
\begin{enumerate}
\item $\pi$ is a flat family of prestable curves of genus $g$.  Included in this is the condition that the sections $\{\sigma_i\}_{1 \leq i \leq n}$ are disjoint from each other and from the nodes of $C_s$;
\item $f$ restricts on each fiber to a degree $d$ map $f_s\colon C_s \to \PP^r$;
\item\label{mC3} the line bundle $\omega_{\pi, \log} \otimes f^*(\cO_{\PP^r}(3))$ is ample, where 
$\omega_{\pi, \log}$ is defined as
\[\omega_{\pi, \log} := \omega_{\pi} \otimes \cO_\cC \left(\sum_{i=1}^n \sigma_i(S) \right).\]

\end{enumerate} 
 There is a natural notion of equivalence of such families.  See \cite{FP} for a construction of $\sMbar_{g, n}(\PP^r, d)$, and a proof that the coarse moduli space is projective.
We argue below that $\sMbar_{g, n}(\PP^r, d)$ can be realized as (a substack of) a space of sections, this time over a smooth Artin stack.  
\begin{definition}\label{d:seco}
Define the contravariant functor $\op{sec}_{\cL^{\oplus r+1}}^\circ\colon \op{Schemes} \to \op{Set}$ which sends a scheme $S$ to the set of 
equivalence classes of families over
 $S$ of the form:
\[
\begin{tikzcd}
 &\cL^{\oplus r+1} \ar[ld] \\
\cC \ar[d, "\pi"] \ar[ur, bend left, "s"] &  \\
S \ar[u, bend left, "\sigma_i"] &
\end{tikzcd}
\]
where:
\begin{enumerate}
\item $\pi$ is a flat family of prestable curves of genus $g$.  The sections $\{\sigma_i\}_{1 \leq i \leq n}$ are disjoint from each other and from the nodes of $C_s$;
\item $\cL$ is a line bundle on $\cC$, of degree  $d$ on each fiber $C_s$;
\item The section $s \in \Gamma(\cC, \cL^{\oplus r+1})$ is nowhere vanishing;
\item \label{iC4} The line bundle $\omega_{\pi, \log} \otimes \cL^{\otimes 3}$ is ample.
\end{enumerate}
A morphism between two such families $(\cC, \cL, s)$ and $(\cC ', \cL ', s') $ over $S$ is given by a morphism of pre-stable curves $\phi\colon \cC \to \cC '$ together with an isomorphism $\cL \to \phi^*(\cL ')$ which identifies $s$ with $\phi^*(s ')$.
\end{definition}
\begin{lemma}\label{l:repr}
$\sMbar_{g, n}(\PP^r, d)$ represents the functor $\op{sec}_{\cL^{\oplus r+1}}^\circ$.
\end{lemma}
This is well-known and follows from the correspondence between maps $T \to \PP^r$ and $r+1$-tuples of sections of a line bundle $L \to T$.  See \cite[Section~2.2]{CFK} for further details and for a generalization to toric varieties.
\begin{proof}[Proof of Lemma~\ref{l:repr}]

Recall that projective space $\PP^r$ is a fine moduli space for the functor which sends a scheme $S$ to the set of equivalence classes of the following data $(L \to S, s \in \Gamma(S, L^{\oplus r+1}))$, where $L \to S$ is a line bundle and $s$ is nowhere vanishing.  Here the collections $(L \to S, s \in \Gamma(S, L^{\oplus r+1}))$ and $(L' \to S, s' \in \Gamma(S, { L '}^{\oplus r+1}))$ are equivalent if there is an isomorphism $L \to L'$ identifying $s$ with $s'$.  

Given a family of curves $\cC \to S$, a family of degree $d$ maps $f\colon \cC \to \PP^r$ corresponds to (an equivalence class of) a line bundle $\cL \to \cC$ and a section $s \in \Gamma(\cC, \cL^{\oplus r+1})$.  Under this equivalence, $\cL$ is identified with $f^*(\cO_{\PP^r}(-1))$.  Thus condition \eqref{iC4} for the functor $\op{sec}_{\cL^{\oplus r+1}}^\circ$ corresponds to condition \eqref{mC3} for the functor of stable maps.
\end{proof}

Forgetting the section $s$, we obtain a morphism from $\sMbar_{g, n}(\PP^r, d)$ to another stack, defined below.
\begin{definition} Define a family of (genus $g$, $n$-marked) prestable curves with a degree $d$ line bundle over $S$
to be given by the following data:
\[
\begin{tikzcd}
\cL \ar[d] \\
\cC \ar[d, "\pi"]   \\
S \ar[u, bend left, "\sigma_i"] 
\end{tikzcd}
\]
where $\cL$ is degree $d$ on each fiber and condition (1) is satisfied.
Define $\Bun_{g, n, d}$, to be the stack representing such families.  

Define $\Bun_{g,n,d}^\circ \subset \Bun_{g,n,d}$ to be the open substack consisting of families such that Condition~\ref{iC4} above is also satisfied.
\end{definition}
  Note that all families in $\Bun_{g, n, d}^\circ$ have, in addition to automorphisms of the underlying curve $\cC$, a  $\CC^*$-family of automorphisms obtained by the scaling automorphisms of $\cL$.
\begin{proposition}\cite[Proposition 2.1.1]{CKM}
The stack $\Bun_{g, n, d}$ is a smooth Artin stack of dimension $4g - 4 + n$.  Over the open locus $\Bun_{g, n, d}^\circ$, the universal curve
 $$\pi\colon \fC \to \Bun_{g, n, d}^\circ$$  is projective.
 \end{proposition}
Note that 
the line bundle $\omega_{\pi, \log} \otimes \cL^{\otimes 3}$
 is relatively ample by construction (see Condition~\ref{iC4} above).  

Consider the universal curve $\pi\colon \fC \to \Bun_{g, n, d}^\circ$, and the universal line bundle $\cL \to \fC$.  Define the vector bundle $\cV := \cL^{\oplus r+1}$  over $\fC$.  

\begin{proposition} \label{p:openinsecs}
There is an open immersion 
\[\sMbar_{g, n}(\PP^r, d) \hookrightarrow \tot(\pi_* \cV).\]

\end{proposition}

\begin{proof}
By Condition~\ref{iC4} of Definition~\ref{d:seco}, $\sMbar_{g, n}(\PP^r, d)$ lies over $\Bun_{g,n,d}^\circ$.
By Proposition~\ref{p:dirimcone} and Lemma~\ref{l:repr}, $\sMbar_{g, n}(\PP^r, d)$ is a substack of $\tot(\pi_* \cV)$, defined by the additional condition that the section $s$ is nowhere vanishing.  We claim that this is an open condition.

Consider the universal curve $\fC_{ \tot(\pi_* \cV)} =  \tot(\pi_* \cV) \times_{\Bun_{g,n,d}^\circ} \fC$ over  $\tot(\pi_* \cV)$.  
On $\fC_{ \tot(\pi_* \cV)}$ there is a universal vector bundle $\cV_{ \tot(\pi_* \cV)} \to \fC_{ \tot(\pi_* \cV)}$ (pulled back from the map $\fC_{ \tot(\pi_* \cV)} \to \fC$), together with a universal section $\tilde s \in \Gamma(\fC_{ \tot(\pi_* \cV)}, \cV_{ \tot(\pi_* \cV)})$.  Consider the closed substack $Z(\tilde  s) \subset \fC_{ \tot(\pi_* \cV)}$ defined by the vanishing of the universal section.  Because the map $\pi\colon \cC \to \Bun_{g,n,d}^\circ$ is proper, so is the map $\pi_{ \tot(\pi_* \cV)} \colon \fC_{ \tot(\pi_* \cV)} \to  \tot(\pi_* \cV)$.  The image $\pi_{ \tot(\pi_* \cV)}(Z(\tilde s))$ is therefore a closed subset of $\tot(\pi_* \cV)$.  Its complement consists of points $(C, \cV|_C \to C, s \in \Gamma(C, \cV_C))$ such that $s$ is nowhere vanishing.  These are exactly the objects of 
Definition~\ref{d:seco}.  We conclude that $\sMbar_{g, n}(\PP^r, d) = \pi_{ \tot(\pi_* \cV)}(Z(\tilde s))^c$ is open in $\tot(\pi_*\cV)$.
\end{proof}

\subsection{The virtual class}\label{ss:vcp}
Finally, we construct a virtual fundamental class for $\sMbar_{g, n}(\PP^r, d)$ via a very minor modification of the procedure in Section~\ref{ss:sos}.  
Let $\cV$ be as in the previous paragraph.  Note that $\pi\colon \fC \to \Bun_{g, n, d}^\circ$ is projective.
As described in Section~\ref{ss:sos}, embed $\cV$ into a $\pi$-acyclic vector bundle $\cA \to \fC$.  Define $\cB$ as the cokernel of $\cV \to \cA$, and let 
\begin{align*}A :=& \pi_*(\cA), \\
B :=& \pi_*(\cB).\end{align*}
Then $[A \to B]$ is a two-term resolution of $\RR \pi_* \cV$ by vector bundles.  Define 
\[\tau\colon \tot(A) \to \Bun_{g, n, d}^\circ\] and $E := \tau^*(B)$ as before. 
Because  $\Bun_{g, n, d}^\circ$ is a smooth Artin stack, $\tot(A)$ will be as well.

For $b \in \Bun_{g, n, d}^\circ$,  the fiber of $\tot(A)$ over $b$ is given by 
\[\tot(A)_b = \{s \in \Gamma(C_b, \cA|_{C_b})\}.\]  Define the open substack
\[U \subset \tot(A)\]
by the condition that the section $s$ is nowhere vanishing on each fiber.  Again this is an open condition.  One can check that $U$ is a Deligne--Mumford stack.  In particular the automorphism group of each point of $U$ is finite.  As described in Section~\ref{ss:sos}, the map $A \to B$ of vector bundles induces a natural section which we denote $\beta$ in $\Gamma(U, E)$.  The zero locus of $\beta$ cuts out exactly those sections of $\cA$ which are sections of $\cV = \cL^{\oplus r+1}$.  As we have restricted to the locus of nowhere vanishing sections, we conclude that 
\[\sMbar_{g, n}(\PP^r, d) = Z(\beta) \subset U \subset \tot(A).\]
\begin{definition}\label{d:vcp}
Define the virtual fundamental class
\[[\sMbar_{g, n}(\PP^r, d)]^{\vir} := e(E, \beta) \in A_*(\sMbar_{g, n}(\PP^r, d)).\]
\end{definition}

As a quick check that this definition is reasonable, we do a dimension count.  The refined Euler class 
$e(E, \beta)$ lies in  $A_k(\sMbar_{g, n}(\PP^r, d))$, where 
\begin{align*}
k &= \dim(\Bun_{g,n,d}) + \rank(A) - \rank(B)\\
& = \dim(\Bun_{g,n,d}) + \chi(V_s) \\
& = 4g-4 +n + (r+1)(1-g+d) \\
& = (1-g)(r-3) + n + (r+1)d\\
& =: \op{virdim}\left( \sMbar_{g, n}(\PP^r, d)\right).
\end{align*}
This is the \emph{expected dimension} of $\sMbar_{g, n}(\PP^r, d)$ \cite[Section 7.1.4]{CK}.  More precisely, there is an open substack $\sMbar_{g, n}(\PP^r, d)^\circ$ of $\sMbar_{g, n}(\PP^r, d)$ parametrizing equivalence classes of maps $f\colon C \to \PP^r$ such that $C$ is smooth and $H^1(C, f^*(\cc O_{\PP^r}(1))) = 0$.  On this open locus $\sMbar_{g, n}(\PP^r, d)^\circ$ is smooth, of dimension $(1-g)(r-3) + n + (r+1)d$.

\begin{proposition}\label{p:agrees}
The virtual fundamental class of Definition~\ref{d:vcp} agrees with the Gromov--Witten theory virtual fundamental class for $\sMbar_{g, n}(\PP^r, d)$ as defined by Behrend and Fantechi in \cite{BF, Beh}.
\end{proposition}
\begin{proof}
Consider the morphism 
\[\phi\colon \sMbar_{g, n}(\PP^r, d) \to \mathfrak M_{g,n}\]
to the stack of pre-stable curves which forgets the map to $\PP^r$ (but does not stabilize the curve).
Following \cite{BF, Beh}, the virtual fundamental class on $\sMbar_{g, n}(\PP^r, d)$ may be defined by 
using the relative perfect obstruction theory $(\RR \pi_*f^* T\PP^r)^\vee$ (see the introduction of \cite{Beh} for details).

The key to the proposition is the following fact,
described in the beginning of Section 6 of \cite{BF}.  Whenever one has the fiber square: 
\begin{equation}\label{e:basicex}
\begin{tikzcd}
Z \ar[d, hookrightarrow] \ar[r, hookrightarrow, "i"] & X \ar[d, hookrightarrow, "s"] \\
X \ar[r, hookrightarrow, "0"] & E,
\end{tikzcd}
\end{equation}
there is a perfect obstruction theory $F^\bullet$ given by the complex
\[ s|_Z^*(E)^\vee \to i^* \Omega_X,\]
together with the natural map to the cotangent complex $L_Z^\bullet$ (see \cite[Section 6]{BF} for details).  In this special case, the virtual fundamental class $[Z, F^\bullet]$ of Behrend--Fantechi is precisely the refined Euler class $e(E,s)$ of Definition~\ref{d:refeu}.  
More generally, consider the relative version of diagram~\eqref{e:basicex} over a base stack $S$.  Assume that $X$ and $S$ are smooth Deligne--Mumford stacks. 
Then the complex \[ s|_Z^*(E)^\vee \to i^* \Omega_{X/S},\]
defines a relative perfect obstruction theory $F_S^\bullet$ for $X/S$.  In this case, again the virtual fundamental class $[Z, F_S^\bullet]$ is the refined Euler class $e(E,s)$. 

On the other hand, as described in Sections~\ref{ss:sos} and~\ref{ss:smssos}, a particular case of diagram~\eqref{e:basicex} is given by:
\begin{equation}\label{e:partdiag}
\begin{tikzcd}
\sMbar_{g, n}(\PP^r, d) \ar[d, hookrightarrow] \ar[r, hookrightarrow, "i"] & U \ar[d, hookrightarrow, "\beta"] \\
U \ar[r, hookrightarrow, "0"] & E.
\end{tikzcd}
\end{equation}
Via the morphism $\nu\colon U \to  \mathfrak M_{g,n}$ which forgets the section of $\cc A $ and the line bundle $\cL$,  diagram \eqref{e:partdiag} is relative to  
$\mathfrak M_{g,n}$.

By the above discussion, the class in Definition~\ref{d:vcp} is equal to the class 
\[[\sMbar_{g, n}(\PP^r, d), F^\bullet_{\mathfrak M_{g,n}}],\] for the relative perfect obstruction theory
\begin{align}  F^\bullet_{\mathfrak M_{g,n}} &=  \tau^* B^\vee \to \Omega_{U/\mathfrak M_{g,n}}.\label{e:POT}
\end{align}
The proposition will follow after identifying the relative perfect obstruction theory $F^\bullet_{\mathfrak M_{g,n}}$ with $(\RR \pi_*f^* T\PP^r)^\vee$.

Consider the following commutative diagram on the universal curve $\cc C$ over $\sMbar_{g, n}(\PP^r, d)$, viewed here as the vanishing locus $Z(\beta) \subset U$:
\begin{equation}\label{e:9}
\begin{tikzcd} 
\cO_{\cC} \ar[r] \ar[d] & \cc V \ar[r] \ar[d] & f^*(T \PP^{r-1}) \ar[d] \\
\cO_{\cC} \ar[r] \ar[d] & \cc A \ar[r] \ar[d] & \cc A ' \ar[d] \\
0 \ar[r] &\cc B  \ar[r] &\cc B.
\end{tikzcd}
\end{equation}
Here $f\colon \cC \to \PP^{r-1}$ is the universal map to $\PP^{r-1}$ and the top row is obtained by pulling back the Euler sequence,
\[0 \to \cO_{\PP^{r-1}} \to \cO_{\PP^{r-1}}(1)^{\oplus r} \to T\PP^{r-1} \to 0,\]
from $\PP^{r-1}$.  
The vector bundles $\cc V$, $\cc A$, and $\cc B$ are all pulled back from the corresponding vector bundles on the universal curve over $\Bun_{g,n,d}^\circ$.
By abuse of notation we will denote them by the same letters.  The section $\cc O_{\cC} \to \cc A$ is defined by the tautological section and the vector bundle $\cc A '$ is the cokernel of this map.  
Note that all rows and columns are short exact sequences.

For a given point $\{C, \cc L|_C \to C, s\in \Gamma(C, \cc A|_C)\} \in U$, the relative tangent bundle $T\nu$ consists of deformations of the line bundle $\cc L|_C$, together with deformations of the section $s\in \Gamma(C, \cc A|_C)$.  The former space is given by $R^1\pi_*(C, \cc O_C)$ and the later is $H^0(C, \cc A|_C) / \CC$, where the quotient by $\CC$ accounts for the fact that two sections give equivalent points of $U$ if they agree up to a scaling.  

The bundle $T\nu$ may be described globally by pushing forward the middle row of the above diagram:
\[\cc O_U \to \tau^*A \to A' \to R^1 \pi_*(\cc O_{\cC}),\]
where $A' := \pi_*(\cc A')$.  This yields the short exact sequence
\[0 \to\op{cok}(\cc O_U \to \tau^*A ) \to A' \to R^1 \pi_*(\cc O_{\cC}) \to 0\]
Deformations of the section $s$ up to scaling are described by exactly $\op{cok}(\cc O_U \to \tau^*A )$ and we observe that $T_\nu = A'$.

Returning to \eqref{e:POT}, $F^\bullet_{\mathfrak M_{g,n}}$ may now be represented as $B^\vee \to {A '}^\vee$, where the map is the dual of the pushforward of $\cc A ' \to \cc B$.  Because the last column of \eqref{e:9} is exact, $F^\bullet_{\mathfrak M_{g,n}}$ is equal to $(\RR \pi_*f^* T\PP^r)^\vee$, and $[\sMbar_{g, n}(\PP^r, d), F^\bullet_{\mathfrak M_{g,n}}]$ is therefore the virtual class of the first paragraph of the proof.


\end{proof}


\section{Stable maps to a hypersurface}\label{s:ne0}
Let $w = w(x_0, \ldots, x_r)$ be a homogeneous polynomial of degree $k$.  We assume that $w$ defines an isolated singularity at the origin in $\CC^{r+1}$, or equivalently, that
the hypersurface $X_w = Z(w) \subset \PP^r$ is smooth.
In this section we expand the ideas of the previous section to construct a virtual class on $\sMbar_{g, n}(X_w, d)$.  

Related to the moduli space of stable maps to $X_w$ is the space of maps to $\PP^r$ with \emph{$p$-fields}, which has the advantage of being a space of sections of a vector bundle as in Section~\ref{ss:sos}.  This moduli space and various generalizations have appeared in \cite{CL, Clader, FJR15}.
\begin{definition}\cite{CL}
Given $k \in \ZZ_{>0}$, define the space of stable maps to $\PP^r$ with a $p$-field of degree $k$ to be the moduli space lying over $\sMbar_{g, n}(\PP^r, d)$ parametrizing families of stable maps $f\colon C \to \PP^r$ of degree $d$ from genus $g$ curves $C$, together with a section
\begin{equation}\label{e:potm} p \in \Gamma(C, f^*(\cO_{\PP^r}(-k) \otimes \omega_{C})).\end{equation}
This moduli stack is denoted by $\sMbar_{g, n}(\PP^r, d)^p$.
\end{definition}
The stack $\sMbar_{g, n}(\PP^r, d)^p$ is easily seen to be a Deligne--Mumford stack because $\sMbar_{g, n}(\PP^r, d)$ is.  We can also view the stack as representing families \[(\pi\colon \cC \to S, \cL \to \cC, (s, p) \in \Gamma(\cC, \cL^{\oplus r+1} \oplus \cL^{\otimes -k}\otimes \omega_{\pi}))\] 
where:
\begin{enumerate}
\item $\pi$ is a flat family of prestable curves of genus $g$;
\item $\cL$ is a line bundle on $\cC$, of degree  $d$ on each fiber $C_s$; 
\item The section $s \in \Gamma(\cC, \cL^{\oplus r+1})$ is nowhere vanishing;
\item The line bundle $\omega_{\pi, \log} \otimes \cL^{\otimes 3}$ is ample.
\end{enumerate}
By the same reasoning as Proposition~\ref{p:openinsecs},
 we see that $\sMbar_{g, n}(\PP^r, d)^p$ is (an open subset of) a space of sections over $\Bun_{g, n, d}^\circ$.

\subsection{Smooth embedding via resolutions}\label{ss:sevr}
We can repeat the argument of the previous section to realize $\sMbar_{g, n}(\PP^r, d)^p$ as the zero locus of a section of a vector bundle defined over a smooth Deligne--Mumford stack.  
Let
\[
\begin{tikzcd}
\cL \ar[d] \\
\cC \ar[d, "\pi"]   \\
S 
\end{tikzcd}
\]
be a family of prestable curves together with a degree $d$ line bundle over $\cC$, pulled back from $\Bun_{g, n, d}^\circ$ via a map $S \to \Bun_{g, n, d}^\circ$.
Define the fiber products 
\begin{align}
\sMbar_{g, n}(\PP^r, d)_S := &S \times_{\Bun_{g, n, d}^\circ} \sMbar_{g, n}(\PP^r, d)\\
\sMbar_{g, n}(\PP^r, d)^p_S := &S \times_{\Bun_{g, n, d}^\circ} \sMbar_{g, n}(\PP^r, d)^p \nonumber \\
\sMbar_{g, n}(X_w, d)_S := &S \times_{\Bun_{g, n, d}^\circ} \sMbar_{g, n}(X_w, d). \nonumber
\end{align}
Let
\begin{align*}\cV_1 := &\cL^{\oplus r+1}, \\
\cV_2 := &\cL^{\otimes -k}\otimes \omega_\pi,\\
\cV := &\cV_1 \oplus \cV_2.\end{align*}
Choose $\pi$-acyclic vector bundles $\cA_1$ and $\cA_2$ together with embeddings of $\cV_1$ and $\cV_2$ to obtain short exact sequences
\begin{align*} 0 \to &\cV_1 \to \cA_1 \to \cB_1 \to 0 \\
0 \to &\cV_2 \to \cA_2 \to \cB_2 \to 0
\end{align*}
where $\cB_1$ and $\cB_2$ are the respective cokernels.  Define
\begin{align*} A:= A_1 \oplus A_2 := &\pi_*(\cA_1) \oplus \pi_*(\cA_2) \\
 B:= B_1 \oplus B_2 := &\pi_*(\cB_1) \oplus \pi_*(\cB_2).\end{align*}
Then the two-term complex of vector bundles $A \to B$ gives a resolution of $\RR \pi_*\cV$ over $S$. 
\begin{definition}\label{d:u}
Define 
$U$ to be the open substack of $\tot(A) = \tot(A_1 \oplus A_2)$,
 where the section $s \in \Gamma(\cC, \cA_1)$ is nowhere vanishing.  Let $\tau\colon U \to \Bun_{g, n, d}^\circ$ be the map which forgets the section of $\cA$, and define 
\[E := \tau^*(B).\]  
\end{definition}The map $A \to B$ induces a section $\beta \in \Gamma(U, E)$.  We conclude that
\[\sMbar_{g, n}(\PP^r, d)^p_S = Z(\beta) \subset U\]
 as desired.  
 
 \subsection{Incorporating $w$}\label{ss:iw}
 
 Recall that $w(x_0, \ldots, x_r)$ is a homogeneous polynomial of degree $k$ defining a hypersurface $X_w \subset \PP^r$. Define 
 \[\hat w = \hat w(x_0, \ldots, x_r, y) := y \cdot w(x_0, \ldots, x_r).\]
 

\begin{proposition}\label{palphap}
The function $\hat w$ induces a morphism \[\tilde \alpha\colon \Sym^{k+1}([A \to B]) \to \cO_S[-1]\] in the derived category of $S$.
\end{proposition}
\begin{proof}
The morphism $\tilde \alpha$ will be defined as a composition of morphisms in the derived category:
\begin{equation}\label{e:map}
\begin{tikzcd}[sep=small]
\Sym^{k+1}([ A \to B]) \ar[r, "\sim"] \ar[dddrrr, dashed, "\tilde \alpha"] & \Sym^{k+1}(\RR \pi_* \cV) \ar[r, "\op{nat}"] &\RR \pi_* (\Sym^{k+1}(\cV)) \ar[r, "\widetilde{\hat w}"] & \RR \pi_*(\omega_\pi) \ar[d] \\
& & & \RR^1 \pi_*( \omega_\pi)[-1] \ar[d] \\
& & & \RR^0 \pi_*(\cO_S)^\vee[-1] \ar[d] \\
 & & & \cO_S [-1].
\end{tikzcd}
\end{equation}
In the above, the first vertical map is obtained by taking cohomology in degree one, the second vertical map is Serre duality, and the third is the trace map.   The first horizontal map is due to the fact that $[ A \to B]$ is equivalent to $\RR \pi_* \cV$.  The second horizontal map
\begin{equation}\label{e:nat}\op{nat}\colon \Sym^{k+1}(\RR \pi_* \cV) \to \RR \pi_* (\Sym^{k+1}(\cV))\end{equation}
 is induced from the composition
\[\mathbb{L}\pi^* \Sym^{k+1}(\RR \pi_* \cV)\to \Sym^{k+1}(\mathbb{L}\pi^* \RR \pi_* \cV) \to \Sym^{k+1}(\cV)\]
via adjunction.  

To define $\tilde \alpha$ it therefore remains to define $\widetilde{\hat w}$.  This is done in Lemma~\ref{l:wmap} below.

\end{proof}

\begin{lemma}\label{l:wmap}
 The function $\hat w$ induces a morphism
 \[ \widetilde{\hat w}\colon \RR \pi_* (\Sym^{k+1}(\cV)) \to \RR \pi_*(\omega_\pi)\]
 in the derived category of $S$.
\end{lemma}
\begin{proof}  This is explained in a general context in Section 3.1.4 of \cite{CFFGKS}.  In this setting it can be seen most easily using local coordinates.  
First, let 
 $$M_j = c_j \prod_{i=1}^k x_{j_i} \cdot y$$
be a monomial of $\hat w$.  
 Given a local section $(s,p) = (s_0, \ldots, s_r, p)$ of $\cV$, note that 
 \[M_j(s,p) := c_j\cdot \bigotimes_{i=1}^k s_{j_i} \otimes p\]
 is a local section of $\omega_\pi$.  Define a map 
 \[\bigotimes_{i=1}^{k+1} \cV \to \omega_\pi\] by sending a section $(s,p)^1 \otimes \cdots \otimes (s,p)^{k+1}$ to 
 \[\sum_{\sigma \in S_{k+1}} \frac{c_j}{(k+1)!}  \bigotimes_{i=1}^k s_{j_i}^{\sigma(i)} \otimes p^{\sigma(k+1)}.\]
 This expression is symmetric in the factors of $\cV$ and therefore induces a map \[\hat M_j\colon \Sym^{k+1}(\cV) \to \omega_\pi.\] Summing over all monomials of $\hat w$ yields a map of vector bundles
\begin{equation}\label{e:Msum} \hat w :=\sum_j \hat M_j\colon  \Sym^{k+1}(\cV) \to \omega_\pi.\end{equation}

Pushing forward via $\pi\colon \cC \to S$, we obtain the desired morphism $$ \widetilde{\hat w}:\RR \pi_* (\Sym^{k+1}(\cV)) \to \RR \pi_*(\omega_\pi)$$ in the derived category of $S$.
  \end{proof}

%

For simplicity of exposition, we will make the following assumption for the remainder of this section.
\begin{assumption}\label{ass1}
Assume that the bundles $A$ and $B$ are such that the map $\tilde \alpha$ exists as a map of \emph{complexes}.  
\end{assumption}
We warn the reader that this assumption may not in fact hold in all cases.  We will explain how to work around it in Section~\ref{s:ts}.

Given Assumption~\ref{ass1}, $\tilde \alpha$ can be represented by a map of chain complexes
\begin{equation}\label{e:cha1}
\begin{tikzcd}
\Sym^{k+1}(A) \ar[d] \ar[r, "d_{k+1}"] & \Sym^k(A) \otimes B \ar[d, "\tilde \alpha"] \ar[r, "d_k"] & \Sym^{k-1} (A)\otimes \wedge^2 (B) \ar[d] \ar[r, "d_{k-1}"] & \cdots \\
0 \ar[r] & \cO_S \ar[r]& 0\ar[r] & \cdots.
\end{tikzcd}
\end{equation}
By abuse of notation we will also denote by $\tilde \alpha\colon \Sym^k(A) \otimes B \to \cO_S$ the only non-zero vertical map.

\begin{proposition}
The differential $d \hat w$ induces a map of $\cO_S$-modules over $S$:
\begin{flalign}\label{a:dw}
\widetilde{d\hat w} \colon  \Sym^k(\pi_*\cV) \otimes \RR^1 \pi_*\cV &\to \cO_S.
\end{flalign}
\end{proposition}
\begin{proof}
By an analogous derivation as in \eqref{e:Msum}, the function $\frac{\partial \hat w}{\partial x_i}$ defines a map of 
$\cO_\cC$-modules:
\[\Sym^k(\cV) \to \cL^\vee \otimes \omega_\pi.\] Tensoring with $\cL$, we obtain a map:
\begin{align}\label{e:dw1}
{\frac{\partial \hat w}{\partial x_i}}\colon \Sym^k(\cV) \otimes \cL  
& \to \omega_\pi.
\end{align}
Similarly, one can define: 
\[{w} = \frac{\partial \hat w}{\partial y} \colon \Sym^k(\cV) \otimes (\cL^{-k}\otimes \omega_\pi) \to \omega_\pi.\]

For $0 \leq i \leq r$ let $\pi_i:\cV \to  \cL$ be the  the projection onto the $i$th summand of $\cV$ and let $\pi_y\colon \cV \to \cL^{-k}\otimes \omega_\pi$ be the projection onto the last summand $\cL^{-k}\otimes \omega_\pi$.  Then 
\begin{equation}\label{e:sdw} d\hat w:= \frac{1}{k+1} \left(\sum_{i=0}^r {\frac{\partial \hat w}{\partial x_i}} \circ \op{id}_{\Sym^k(\cV)} \otimes \pi_i +  w \circ \op{id}_{\Sym^k(\cV)} \otimes \pi_y \right)\end{equation} gives a map of vector bundles
$\Sym^k(\cV) \otimes \cV \to \omega_\pi$.  Pushing forward via $\pi\colon \cC \to S$, we obtain a morphism \[\widetilde{d\hat w}\colon \RR \pi_* (\Sym^{k}(\cV) \otimes \cV) \to \RR \pi_*(\omega_\pi)\] in the derived category of $S$.

Composing with the natural map (defined in an analogous manner to \eqref{e:nat})
\[\op{nat}\colon \Sym^{k} (\RR \pi_* \cV) \otimes \RR \pi_*\cV \to \RR \pi_* (\Sym^{k}(\cV)) \otimes \RR \pi_* \cV \to \RR \pi_* (\Sym^{k}(\cV) \otimes \cV) \]
and the vertical maps in \eqref{e:map}, we obtain a morphism
\[\Sym^{k} (\RR \pi_* \cV) \otimes \RR \pi_*\cV \to \cO_S[-1].\]

By composing this with the morphism 
\[\Sym^{k} ( \pi_* \cV) \otimes \RR \pi_*\cV \to \Sym^{k} (\RR \pi_* \cV) \otimes \RR \pi_*\cV\]
and applying $H^1$ to the composition we obtain a map of $\cO_S$-modules 
\begin{equation}\label{e:dw1}H^1\left(\Sym^{k} ( \pi_* \cV) \otimes \RR \pi_*\cV\right) \to \cO_S.\end{equation}
There is a canonical map 
\begin{equation}\label{e:dw2}\Sym^{k} ( \pi_* \cV) \otimes \RR^1 \pi_*\cV \to H^1\left(\Sym^{k} ( \pi_* \cV) \otimes \RR \pi_*\cV\right).\end{equation}
The composition of \eqref{e:dw1} with \eqref{e:dw2} gives the desired map. 


\end{proof}

\begin{proposition}\cite[Lemma 3.6.3]{CFFGKS} \label{p:supp0}
The following diagram commutes
\[
\begin{tikzcd}
\Sym^k( \pi_* \cV) \otimes B \ar[dr, two heads, "\op{id}\otimes H^1(-)"] \ar[rr, "\tilde \alpha"] && \cO_S \\
&\Sym^k(\pi_*\cV) \otimes \RR^1 \pi_*\cV \ar[ur, "\widetilde{d \hat w}"]&
\end{tikzcd}
\] 
where $H^1(-)$ is the map $B \to B/A = \RR^1 \pi_*\cV$.
\end{proposition} 

\begin{proof}
The proof rests on Euler's homogeneous function theorem, which states that 
\[\hat w = \frac{1}{k+1} \left( \sum_{i=0}^r x_i \cdot \frac{\partial}{\partial x_i} \hat w + y \cdot w\right).\]
A direct calculation on local sections then shows that the following diagram commutes
\begin{equation}\label{e:Euler2}
\begin{tikzcd}
\Sym^{k}\cV \otimes \cV \ar[dr, "d\hat w"]  \ar[r, "m"] & \Sym^{k+1}(\cV) \ar[d, "\hat w"]\\
& \omega_\pi,
\end{tikzcd}
\end{equation}
where $d\hat w$ and $\hat w$ are defined in \eqref{e:sdw} and \eqref{e:Msum} and $m$ is the map 
$(v_1 \otimes \cdots \otimes v_k) \otimes v \mapsto v_1 \otimes \cdots \otimes v_k \otimes v$.

Consider the following commutative diagram
\begin{equation}\label{dh1}
\begin{tikzcd}
\Sym^k(\pi_*\cV) \otimes B[-1] \ar[d] & \\
\Sym^k(\pi_*\cV) \otimes \RR \pi_*\cV \ar[d] & \\
\Sym^k(\RR \pi_*\cV) \otimes \RR \pi_*\cV \ar[d]  \ar[rd]& \\
\RR \pi_*(\Sym^k(\cV)) \otimes \RR \pi_*\cV \ar[d] & \Sym^{k+1}(\RR \pi_*\cV) \ar[d] \\
\RR \pi_*(\Sym^k(\cV) \otimes \cV) \ar[dr] \ar[r] &  \RR \pi_*(\Sym^{k+1}(\cV)) \ar[d] \\
& \RR \pi_*(\omega_\pi) \ar[d]\\
& \cO_S[-1].
\end{tikzcd}
\end{equation}
Commutativity of the lower triangle is a result of \eqref{e:Euler2}.
Commutativity of the upper trapezoid follows from the commutativity of the diagram below together with adjunction and the projection formula.
\begin{equation}
\begin{tikzcd}
 & \LL \pi^* \Sym^k(\RR\pi_*\cV) \otimes \LL \pi^* \RR \pi_*\cV \ar[d]& \\
&  \Sym^k(\LL \pi^* \RR \pi_*\cV) \otimes \LL \pi^* \RR \pi_*\cV \ar[dl] \ar[dd] \ar[dr]  & \\
\Sym^k(\cV) \otimes \LL \pi^* \RR \pi_*\cV \ar[dr] & & \Sym^{k+1}(\LL \pi^* \RR \pi_*\cV) \ar[d] \\
& \Sym^k(\cV) \otimes \cV \ar[r] & \Sym^{k+1}(\cV).
\end{tikzcd}
\end{equation}
Applying $H^1$ to  \eqref{dh1}, the left path becomes $\widetilde{d \hat w} \circ \op{id}\otimes H^1(-)$ and the right path becomes $\tilde \alpha$.  The proposition follows.
\end{proof}

\begin{definition}\cite[Definition~3.3]{CL}
Let $\cc F$ be a coherent sheaf on a stack $S$.  A \emph{cosection} of $\cc F$ is a homomorphism of $\cO_S$-modules 
\[ \sigma\colon \cc F \to \cO_S.\]
Given a cosection $\sigma\colon \cc F \to \cO_S$, we denote by $\{\sigma \equiv 0\}$ the locus 
\[ \{ \sigma \equiv 0\} = \Big\{ s \in S \; \Big| \; \sigma|_s\colon \cc F \otimes_{\cO_S} \mathbf k(s) \to \mathbf k(s) \text{ vanishes}\Big\}.\]
\end{definition}

The map of vector bundles $\tilde \alpha\colon \Sym^k(A) \otimes B \to \cO_S$ defines a cosection 
\[\alpha\colon E= \tau^*(B) \to \cO_{\tot(A)}\]
 as follows.  Given a local section $b$ of $E$, define 
\[\alpha(b) := \tau^*(\tilde \alpha)\left( \frac{a ^k}{k!} \otimes b\right) \in \cO_{\tot(A)},\]
where $a \in \Gamma(\tot(A), \tau^*(A))$ is the tautological section.  Here we view $\frac{a ^k}{k!}$ as the tautological section of $\Sym^k \tau^*(A)$.
%
Over $U$ we now have the following maps of $\cO_U$-modules,
\[
\begin{tikzcd}
 \cO_{U} \ar[r, "\beta"] & \ar[d] E \ar[r, "\alpha"] & \cO_{U}\\
& U. & 
\end{tikzcd}
\]

The map $\widetilde{d\hat w}$  defines a cosection of $ \tau^* \left(  \RR^1 \pi_* \cV \right)$ over $\sMbar_{g, n}(\PP^r, d)^p_S$ which, by abuse of notation, we also denote by $d\hat w$:
 \begin{flalign} d\hat w\colon \tau^* \left(  \RR^1 \pi_* \cV \right) \to \cO_{\sMbar_{g, n}(\PP^r, d)^p_S}
 \end{flalign}
 The cosection $ d\hat w$ can be written in coordinates as:
 \begin{flalign}
 (\dot s, \dot p) & \mapsto \frac{1}{k+1} \left(  \frac{1}{k!}\sum_{i=0}^r \frac{\partial \hat w}{\partial x_i}(s_0, \ldots, s_r, p) \dot s_i + \frac{1}{k!}w(s_0, \ldots, s_r)\dot p \right).
\nonumber
\end{flalign}
where $(s,p) = (s_0, \ldots, s_r, p)$ is the tautological section of $ \pi_* \cV$ on $\tot(\pi_* \cV)$
 and $(\dot s, \dot p)$ is a section of $\tau^* \left(  \RR^1 \pi_* \cV \right)$.
 The expression on the right is then an element of $\RR^1(\omega_\pi) \cong \cO_{\tot(\cV)}$.
  (see \cite[Section~3.3]{CL}).

 Since $w$ defines an isolated singularity at the origin, the locus 
 \[\Big\{(x_0, \ldots, x_r) \;\Big|\; w(x_0, \ldots, x_r) = 0, \frac{\partial w}{\partial x_i}(x_0, \ldots, x_r) = 0\Big\}\]
 is the origin.  Consequently,
  the cosection $d \hat w$ is identically zero only when $p = 0$ and $w(s_0, \ldots, s_r) = f^*(w) = 0$. 
But  $\sMbar_{g, n}(X_w, d)_S \subset \sMbar_{g, n}(\PP^r, d)_S \subset \sMbar_{g, n}(\PP^r, d)^p_S$ is exactly the locus where $p$ and $f^*(w)$ vanish.  We conclude that:
\[ \{d\hat w \equiv 0\}  = \sMbar_{g, n}(X_w, d)_S.\]

\begin{corollary}\label{c:sup1}
The closed substack $Z(\beta) \cap \{\alpha \equiv 0\}$ is equal to $\sMbar_{g, n}(X_w, d)_S \subset U$.
\end{corollary}

\begin{proof}
Consider the cosection $\alpha\colon E \to \cO_{U}$ of $E$ on $U$.  After restricting to $Z(\beta)$ and identifying this locus with $\sMbar_{g, n}(\PP^r, d)^p_S$, Proposition~\ref{p:supp0}  implies that on $Z(\beta)$, the cosection $\alpha$ factors as a surjection followed by $d\hat w$.  Thus the locus $Z(\beta) \cap \{\alpha \equiv 0\}$ equals $\{d\hat w \equiv 0\}= \sMbar_{g, n}(X_w, d)_S$. 
\end{proof}

\subsection{$\ZZ_2$-localized Chern character}\label{ss:kf0}
In Section~\ref{ss:zerolocus} we saw that given a vector bundle with a section, one could construct a refined Euler class supported on the vanishing locus of the section.  In this section we describe how one can generalize this construction to also incorporate the data of a cosection.  In this case the refined Euler class will be supported on the intersection of the vanishing locus of the section and the vanishing locus of the cosection.  This is desirable because, as we have seen, the vanishing locus of just the section may not be compact, whereas the intersection of the vanishing locus of the section and the vanishing locus of the cosection may be compact.  Having a virtual class supported on a compact space allows us to define numerical invariants, by integrating over the virtual class.

The generalization of the refined Euler class to incorporate a cosection was first described in \cite{PVold}.  The key is a modification of \emph{MacPherson's graph construction} \cite{Fu, Mac}.

Given a smooth variety $X$ and a vector bundle $E$ of finite rank,  the following identity holds \cite[Example 3.2.5]{Fu}
\begin{equation}\label{e:EC1} e(E) = \ch(\wedge^\bullet E^\vee) \td(E)\end{equation}
where $\ch$ is the Chern character, $\td$ is the Todd class (viewed as an element of $A_*(X)$), and $\wedge^\bullet E^\vee$ is the class in $K$-theory given by the alternating sum \[\oplus_{k = 0}^{\op{rk}(E)} (-1)^k \wedge^k E^\vee.\] 
This identity can be refined to include a section $\sigma \in \Gamma(X, E)$.  Define the Koszul complex
\[ K := \wedge^{\op{rk}(E)} E^\vee \to \cdots \to \wedge^2 E^\vee \to E^\vee \to \cO_X,\]
where the differential is given by contraction with respect to $\sigma$, $d(-) := - \lrcorner \sigma$.  
Since $\sigma$ is non-vanishing outside of $Z = Z(\sigma)$,  a standard linear algebra exercise shows this complex is exact outside of $Z(\sigma)$, i.e. $K$ is supported on $Z$.
In this setting,  MacPherson's graph construction describes
 a limiting procedure similar to that in Definition~\ref{conelim} 
 (see \cite[Section~18.1]{Fu} and \cite{Mac} for details).  This 
defines a \emph{localized Chern character}: \[\ch^X_Z(K)\colon A_*(X) \to A_*(Z).\]  
The following identity holds
\begin{align}\label{e:EC2}
e(E, \sigma) = \ch^X_Z(K)(\td(E)),
\end{align}
where $ e(E, \sigma)$ is the refined Euler characteristic of Definition~\ref{d:eref}.

An insight of Polishchuk--Vaintrob in \cite{PVold} was to adapt the definition of the localized Chern character to the case of two-periodic complexes.  Given an infinite two-periodic complex $K$ of vector bundles, let $Z$ denote the support of $K$.  
In \cite{PVold} a \emph{$\ZZ_2$-localized Chern character} is defined: \[ {}^{\ZZ_2} \ch^X_Z(K)\colon A_*(X) \to A_*(Z).\] 
The construction is a $\ZZ_2$-graded version of the original construction.

\begin{definition}\label{d:kf0}
Assume we have a vector bundle $E \to X$  with a section $\sigma \in \Gamma(X, E)$ and a cosection $\rho \in \hom(E, \cO_X)$ such that $\rho ( \sigma) = 0$.
Define 
the vector bundles \begin{align*}
F_0 &:= \oplus_{k} \wedge^{2k} E^\vee \\
F_{-1} &:= \oplus_{k}  \wedge^{2k+1} E^\vee,
\end{align*}
and define a differential between them $d(-):= -\wedge \rho + -\lrcorner \sigma$. One can check that $d_{i\pm1} \circ d_i = \op{id}_{F_i} \cdot \rho ( \sigma) = 0$.  Denote by $\{\rho, \sigma\}$ the two-periodic complex of vector bundles:
\[\{\rho, \sigma\}:= \cdots \to F_{-1} \xrightarrow{d_{-1}} F_0 \xrightarrow{d_0} F_{-1} \xrightarrow{d_{-1}} F_0 \to \cdots.\]
\end{definition}
It is shown in \cite[Lemma~1.4.1]{PV} that the support of $\{\rho, \sigma\}$ is $Z(\sigma) \cap \{\rho \equiv 0\}$. 
The complex $\{\rho, \sigma\}$ may be viewed as a two-periodic generalization of the Koszul complex to include a cosection; we will to refer it as a \emph{two-periodic Koszul complex}.

\subsection{The virtual class}\label{ss:vc0}
Let us now return to the setting of Section~\ref{ss:iw}.
Given Assumption~\ref{ass1}, we have a vector bundle $E \to U$, together with a section $\beta \in \Gamma(U, E)$ and a cosection $\alpha \in \hom(E, \cO_{U})$.  
\begin{proposition}\cite[Lemma~4.2.6]{PV} \label{p:compz}
The function $\alpha( \beta)$ is identically zero.
\end{proposition}
\begin{proof}
Recall from Definition~\ref{d:u}, that $U$ is an open subset of $\tot(A)$, $A \xrightarrow{\delta} B$ is a two term complex of vector bundles on $S$,  $E$ is the pullback of $B$ via the map $\tau\colon \tot(A) \to S$, and $\beta \in \Gamma(\tot(A), E)$ is the section induced by $\delta$.  

First note that the section $\beta $ is equal to $ \tau^*\delta(a)$, 
where $a \in \Gamma(\tot(A), \tau^*(A))$ is the tautological section.  The composition
$\alpha(\beta)$ is given by  $\tau^* \tilde \alpha \left(\frac{a^k}{k!} \otimes \tau^*\delta(a)\right)$, where $\tilde \alpha$ is defined in \eqref{e:cha1}.  On the other hand, $\frac{a^k}{k!} \otimes \tau^*\delta(a)$ is equal to $\tau^* d_{k+1} \left(\frac{a^{k+1}}{(k+1)!}\right)$.  By the commutativity of the first square of \eqref{e:cha1}, 
\[\alpha(\beta)  =\tau^*( \tilde \alpha \circ d_{k+1})\left(\frac{a^{k+1}}{(k+1)!}\right) = 0.\]

%
%

\end{proof}
By Proposition~\ref{p:compz}, we can construct the two-periodic Koszul complex $\{\alpha, \beta\}$ as above.  
The support of $\{\alpha, \beta\}$ is given by $Z(\beta) \cap \{\alpha \equiv 0\}$ which,
by Corollary~\ref{c:sup1}, is $ \sMbar_{g, n}(X_w, d)_S$.  In analogy with \eqref{e:EC2}, we make the following definition.
\begin{definition}\label{d:vcn0} Let 
$\sM_{X_w} := \sMbar_{g, n}(X_w, d)_S$.
Define the virtual class
\[
\left[ \sMbar_{g, n}(\PP^r, d)^p_S\right]^{\vir} := {}^{\ZZ_2} \ch^{U}_{\sM_{X_w}}(\{\alpha, \beta\}) (\td(E)) \in A_*(\sMbar_{g, n}(X_w, d)_S).\]
\end{definition}

\begin{remark}\label{r:vc0}
If Assumption~\ref{ass1} held for $S = \Bun_{g, n, d}^\circ$, we would obtain a virtual class 
\[
\left[ \sMbar_{g, n}(\PP^r, d)^p\right]^{\vir} \in A_*(\sMbar_{g, n}(X_w, d)).
\]
 Unfortunately Assumption~\ref{ass1} may not hold in this case.  In Section~\ref{s:ts} we explain how to overcome this difficulty.
\end{remark}

\section{Technical assumptions and the ``two-step procedure''}\label{s:ts}
The construction of a virtual class on $\sMbar_{g, n}(X_w, d)$ using the method of Section~\ref{ss:vc0} would require 
Assumption~\ref{ass1} holding for $\Bun_{g,n,d}^\circ$.  This is not known, and likely not true.  In this section we describe a ``two-step procedure'' to circumvent the problem.  For details and a general formulation see Section~4 of \cite{CFFGKS}.

\subsection{Step 1: a projective embedding}
Fix $g, n ,d \geq 0$ such that $2g-2 + n > 0$.
We first embed $\sMbar_{g,n}(\PP^{r}, d)$ into (the smooth locus of) a larger space of stable maps.

Consider the stabilization map:
\[\st\colon \Bun_{g,n,d}^\circ \to \sMbar_{g,n},\]
which forgets the line bundle $\cL$ and stabilizes the curve.  Let $\tilde \st\colon \fC \to \cC_{\st}$ denote the map between the universal curves over $\Bun_{g,n,d}^\circ \to \sMbar_{g,n}$.  The coarse underlying space $\uC_{\st}$ associated to the Deligne--Mumford stack $\cC_{\st}$ is projective, and so can be embedded into $\PP^{N-1}$.
We obtain the following diagram:
\[
\begin{tikzcd}
\cL \ar[d] & & \cO_{\PP^{N-1}}(1) \ar[d] \\
\fC \ar[r, "\tilde{\st}"] \ar[d, "\pi"] & \cC_{\st} \ar[r,"\rho"] \ar[d] &  \uC_{\st} \subset \PP^{N-1}\\
\Bun_{g,n,d}^\circ \ar[r, "\st"] & \sMbar_{g,n}. &
\end{tikzcd}
\]

Let $\cN\to \fC$ denote the line bundle $\cN = \tilde{\st}^* \circ \rho^* ( \cO_{\PP^{N-1}}(1))$.  One can choose an appropriate projective embedding $ \uC_{\st} \subset \PP^{N-1}$ so that $\cL \otimes \cN$ is $\pi$-acyclic.  Consider the Euler sequence
\[ 0 \to \cO_{\PP^{N-1}} \to \cO_{\PP^{N-1}}(1)^{\oplus N} \to T\PP^{N-1} \to 0.\]
Pulling this back to $\fC$ and tensoring with $\cV_1 = \cL^{\oplus r+1}$, we obtain
\[0 \to \cV_1 \to \cM^{\oplus M} \to \cQ \to 0\]
where $\cM := \cL \otimes \cN$, $M = (r+1)N$ and $\cQ = \cV_1 \otimes \tilde{\st}^* \circ \rho^*(T\PP^{N-1})$.
Pushing forward via $\pi$ we obtain the long exact sequence on $\Bun_{g,n,d}^\circ$:
\begin{align}
0 \to & \pi_*(\cV_1) \to  \pi_* \left(\cM^{\oplus M}\right) \to  \pi_* (\cQ) \\
\to & \RR^1 \pi_*(\cV_1) \to 0 \nonumber
\end{align}
where the last terms are zero because $\cM$ was constructed to be $\pi$-acyclic.
Letting
\begin{align*}A_1 ' := & \pi_* \left(\cM^{\oplus M}\right), \\ B_1 ' := &\pi_* (\cQ),\end{align*}
we see that $[A_1 ' \to B_1 ']$ is a two term resolution of $\RR\pi_*(\cV_1)$ by vector bundles.
We also obtain an embedding $\tot(\pi_* \cV_1)$ in $\tot(A_1 ')$.  Let 
\[\tau\colon \tot(A_1 ') \to \Bun_{g,n,d}^\circ\]
 denote the map forgetting the sections.  Then $\tot(\pi_*\cV_1)$ is the zero locus of the section $\beta_1 \in \Gamma(\tot(A_1 '), \tau^*(B_1 '))$ induced by the map $A_1 ' \to B_1 '$.

The section of $\cN^{\oplus N}$ obtained by pulling back the first terms of the Euler sequence is nowhere vanishing.  Therefore nowhere vanishing sections of $\cV_1$  are mapped to nowhere vanishing sections of $\cM^{\oplus M}$.  
Define $U_1$ to be the open substack of $\tot(A_1 ')$ consisting of sections of $\cM^{\oplus M}$ which are nowhere vanishing.  Note that these sections correspond to a map to $\PP^{M-1}$ of some degree $e$.  This realizes $U_1$ as an open subset of $\sMbar_{g,n}(\PP^{M-1}, e)$.  By construction $U_1$ is a smooth Deligne--Mumford stack, lying in the smooth 
locus of $\sMbar_{g,n}(\PP^{M-1}, e)$.  Let $\overline U_1$ denote the closure of $U_1$ in $\sMbar_{g,n}(\PP^{M-1}, e)$.
The zero locus of $\beta_1|_{U_1}$ is equal to $\sMbar_{g,n}(\PP^r,d)$.  We have
\[\sMbar_{g,n}(\PP^r,d) \hookrightarrow U_1 \subset \overline U_1 \subset \sMbar_{g,n}(\PP^{M-1}, e).\]

\subsection{Step 2: an admissible resolution}\label{ss:s2}
The next step is to construct a second resolution, this time of $\tau^*(\RR \pi_*\cV)$ over $U_1$.  Because $\overline U_1$ has projective coarse moduli space, this can be done in such a way that the map $\tilde \alpha$ is realized at the level of complexes.

Let $\pi_{U_1}\colon \cC_{U_1} \to U_1$ and $\pi_{\overline U_1} \colon \cC_{\overline U_1} \to \overline U_1$ denote the universal curves over $U_1$ and $\overline U_1$ respectively.  Let $\cL_{U_1}$, $\cN_{U_1}$, and $\cM_{U_1}$ denote the pullback of $\cL$, $\cN$, and $\cM$ from $\fC$ to $\cC_{U_1}$.  These line bundles naturally extend to line bundles  $\cL_{\overline U_1}$, $\cN_{\overline U_1}$, and $\cM_{\overline U_1}$ over $\overline U_1$ as follows.  Since $\overline U_1 \subset \sMbar_{g,n}(\PP^{M-1}, e)$, there exists a universal map $f\colon \cC_{\overline U_1} \to \PP^{M-1}$.  Define
$\cM_{\overline U_1}$ to be the pullback of $\cO_{\PP^{M-1}}(1)$.  On the other hand by forgetting $f$ and stabilizing the curve, we get a map $\overline U_1 \subset \sMbar_{g,n}(\PP^{M-1}, e) \to \sMbar_{g, n}$, and consequently a map
\[ \cC_{\overline U_1} \to \cC_{\st} \to \uC_{\st} \hookrightarrow \PP^{N-1}.\]
Define $\cN_{\overline U_1}$ to be the pullback of $\cO_{\PP^{N-1}}(1)$.  Finally, define 
\[\cL_{\overline U_1} := \cM_{\overline U_1} \otimes \cN^\vee_{\overline U_1}.\]
The line bundle $\cL_{\overline U_1} \to \cC_{\overline U_1}$ is a degree $d$ line bundle over the family $\cC_{\overline U_1} \to \overline U_1$ of pre-stable curves, therefore there is an induced map $\overline U_1 \to \Bun_{g, n, d}$, which, when restricted to $U_1 \subset \overline U_1$, recovers the forgetful map $U_1 \to \Bun_{g, n ,d}^\circ \subset \Bun_{g, n, d}$. 

Let \begin{align*}\cV_{\overline U_1, 1} := &\cL_{\overline U_1}^{\oplus r+1},\\ \cV_{\overline U_1, 2}  := & \cL_{\overline U_1}^{\otimes -k} \otimes \omega_{\pi_{\overline U_1}},\end{align*} 
and define \[\cV_{\overline U_1} := \cV_{\overline U_1, 1}  \oplus \cV_{\overline U_1, 2} .\]
Let 
$\cV_{ U_1, 1}, \cV_{ U_1, 2},$ and $\cV_{U_1}$
denote the corresponding restrictions to $\cC_{U_1}$.

Using the fact that $\overline U_1$ has projective coarse underlying space,  we can resolve $\cV_{\overline U_1}$ by a two term complex of vector bundles for which the map $\tilde \alpha$ is realized at the level of complexes.
\begin{proposition}\cite[Proposition 3.5.2]{CFFGKS}\label{p:maps}
On $\overline U_1$, there exist resolutions 
\begin{align*}[A_{\overline U_1, 1} '' \to B_{\overline U_1, 1} ''] \sim &\RR \pi_*(\cV_{\overline U_1, 1} ), \\ [A_{\overline U_1, 2} '' \to B_{\overline U_1, 2} ''] \sim & \RR \pi_*(\cV_{\overline U_1, 2} )\end{align*}
such that, if we let \[[A_{\overline U_1} '' \to B_{\overline U_1} ''] := [A_{\overline U_1, 1} '' \oplus A_{\overline U_1, 2} '' \to B_{\overline U_1, 1} '' \oplus B_{\overline U_1, 2} ''],\]
there exists
a map
\[\tilde \alpha\colon \Sym^k (A_{\overline U_1} '') \otimes B_{\overline U_1} '' \to \cO_{\overline U_1}\]
fitting into the diagram \eqref{e:cha1} on $\overline U_1$ and realizing the map from Proposition~\ref{palphap} at the level of complexes.
\end{proposition}
%

Let \begin{align*}A_{U_1, 1}'' \to & B_{U_1, 1}'', \\ A_{U_1, 2}'' \to &B_{U_1, 2}'',\end{align*} denote the restrictions of \begin{align*}A_{\overline U_1, 1} '' \to &B_{\overline U_1, 1} '', \\ A_{\overline U_1, 2} '' \to & B_{\overline U_1, 2} '', \end{align*} to $U_1$.
Let \begin{align*} A_{U_1, 1}' := &\tau^*(A_1 '), \\ B_{U_1, 1}' := &\tau^*(B_1 ').\end{align*} In the derived category $D(U_1)$ we have the equivalence
\[[A_{U_1, 1}'' \to B_{U_1, 1}''] \sim [A_{U_1, 1}' \to B_{U_1, 1}'].\]
By standard arguments (see Lemma~3.6.5 of \cite{CFFGKS}) there exists another resolution $[A_{U_1, 1}   \to B_{U_1, 1}  ]$ of $\RR \pi_*(\cV_{ U_1, 1} )$ and a roof diagram 
\begin{equation}\label{e:roofdiagram}
\begin{tikzcd}
& A_{U_1, 1}  
 \ar[ld, twoheadrightarrow, swap] \ar[r, " d_1"] \ar[rd, twoheadrightarrow] & B_{U_1, 1} \ar[rd, twoheadrightarrow] \ar[ld, twoheadrightarrow, crossing over, near end, swap] 
& \\
A_{U_1, 1}'' \ar[r, "d_1'' "] & B_{U_1, 1}''  & A_{U_1, 1}'  \ar[r, "d_1 ' "]  
  & B_{U_1, 1}'
\end{tikzcd}
\end{equation}
where the diagonal maps of two-term complexes are quasi-isomorphisms.  Consider the resolution of $\RR \pi_*(\cV_{U_1})$ given by
\[ [A_{U_1} \to B_{U_1}] := [A_{U_1, 1}\oplus A_{U_1, 2}''   \to  B_{U_1, 1} \oplus B_{U_1, 2} ''].\]
By composing the map $\tilde \alpha$ of Proposition~\ref{p:maps} with the left diagonal map of \eqref{e:roofdiagram}, we obtain
the map (which by abuse of notation we will also denote as $\tilde \alpha$)
\begin{align} \label{e:mapsf}
\tilde \alpha\colon &\Sym^k (A_{ U_1} ) \otimes B_{ U_1}  \to \cO_{ U_1}.
\end{align}
This map fits  into  the diagram \eqref{e:cha1} over $U_1$ and and realizes the map from Proposition~\ref{palphap} at the level of complexes.

%

Consider the morphism 
\[\tilde \tau\colon \tot(A_{U_1}) \to U_1\]
forgetting the section of $A_{U_1}$.  Define $E := \tilde \tau^*(B_{U_1})$.  The map $[A_{U_1} \to B_{U_1}]$ defines a section
\[\beta \in \Gamma(\tot(A_{U_1}), E)\]
and the map $\tilde \alpha$ of \eqref{e:mapsf} defines a cosection
\[\alpha \in \hom(E, \cO_{\tot(A_{U_1})}).\]
The function $\alpha(\beta)$ is equal to zero by Proposition~\ref{p:compz}.
We can therefore define the two-periodic complex $\{\alpha, \beta\}$ as in Definition~\ref{d:kf0}.


\subsection{The cut-down procedure}

The space $\tot(A_{U_1})$ lies over $U_1 \subset \tot(A_1 ')$, so the relative dimension of $\tot(A_{U_1}) \to \Bun_{g,n,d}^\circ$ is $\rank(A_{U_1}) + \rank(A_1 ')$.  If we view $E= \tilde \tau^*(B_{U_1})$ as an obstruction bundle, the relative virtual dimension over $ \Bun_{g,n,d}^\circ$ is 
\[\rank(A_{U_1}) - \rank(B_{U_1}) + \rank(A_1 ') = \chi(\RR \pi_*\cV)+ \rank(A_1 ')\]
which is too large by $\rank(A_1 ')$.  This overcounting is due to the fact that we resolved $\RR\pi_*(\cV_1)$ twice.\footnote{More precisely,  $[A_1' \to B_1']$ resolved $\RR \pi_*(\cV_1)$, and $[A_{U_1, 1} \to B_{U_1, 1}]$ resolved $\tau^*(\RR \pi_*(\cV_1))$.}
We must correct for this redundancy by choosing an appropriate closed subset of codimension equal to $\rank(A_1 ')$.

We have tautological sections $\taut \in \Gamma(\tot(A_{U_1}), \tilde \tau^*(A_{U_1, 1}))$ and $\taut ' \in \Gamma(U_1, A_{U_1,1}')$.  Let $f_1\colon A_{U_1, 1} \to A_{U_1, 1} '$ denote the surjective map from \eqref{e:roofdiagram}.  Consider the section
\[\xi := \tilde \tau^*(f_1) \circ \taut - \tilde \tau^*(\taut ') \in \Gamma(\tot(A_{U_1}), \tilde \tau^*(A_{U_1, 1} ')).\]
\begin{definition}
Define the substack of $\tot(A_{U_1})$ 
\[\square := \{\xi = 0\}.\]
\end{definition}
By definition of $\xi$, $\square$ consists of triples $(a_1 ', a_1, a_2)$ where $a_1 ' \in U_1 \subset \tot(A_1 ')$, $a_1 \in \tot(A_{U_1,1})|_{a_1 '}$, and $a_2 \in \tot(A_{U_1,2})|_{a_1 '}$ such that $f_1(a_1) = a_1 ' $.  Because $f_1$ is surjective, $\square$ will be smooth of relative dimension $\rank(A_{U_1})$ over $\Bun_{g, n, d}^\circ$.  

Let 
\[\iota\colon \square \hookrightarrow \tot(A_{U_1})\] denote the inclusion.  
\begin{definition}
Define the \emph{Fundamental two-periodic Koszul complex}
\[K_{g, n, d} := \iota^*(\{\alpha, \beta\}).\]
\end{definition}
The terms of $K_{g, n, d}$ consist of the wedge powers of $\iota^*(E)$.

\begin{proposition} The support of $K_{g,n,d}$ is contained in 
$ \sMbar_{g, n}(X_w, d).$
\end{proposition}
\begin{proof}
By \cite[Lemma~1.4.1]{PV}, 
\[\op{supp}\left( K_{g,n,d} \right) = Z(\iota^*(\beta)) \cap \{\iota^*(\alpha)
 \equiv 0\}. \]  The locus $Z(\iota^*(\beta))$ is equal to 
 $ \sMbar_{g, n}(\PP^r, d)^p$
after restricting  to $\square$ in $\tot(A_{U_1})$ .   By Corollary~\ref{c:sup1}, the intersection of this with $\{\iota^*(\alpha)
 \equiv 0\}$ is equal to  
  \[\sMbar_{g, n}(\PP^r, d)^p \cap \{d\hat w \equiv 0\} =\sMbar_{g, n}(X_w, d).\]

\end{proof}

\begin{definition}\label{d:vcf} Let 
$\sM_{X_w} = \sMbar_{g, n}(X_w, d)$.
Define the virtual class
\[
\left[ \sMbar_{g, n}(\PP^r, d)^p \right]^{\vir} := {}^{\ZZ_2} \ch^{\square}_{\sM_{X_w}}(K_{g,n,d}) (\td(\iota^*(E))) \in A_*(\sMbar_{g, n}(X_w, d)).\]
\end{definition}
Definition~\ref{d:vcf} gives a new construction of the virtual fundamental class on the space of stable maps to a hypersurface, using a generalization of the refined Euler class to two-periodic complexes.
We state without proof two comparison results to show that this class agrees with previous constructions of the virtual fundamental class.

\begin{proposition}\cite[Proposition 6.1.7]{CFFGKS}
The virtual class $\left[ \sMbar_{g, n}(\PP^r, d)^p \right]^{\vir}$ agrees with the \emph{cosection-localized} virtual fundamental class defined by Chang--Li in \cite{CL}.
\end{proposition}
The proof uses results of \cite{CLL} together with the compatibility of $\tilde \alpha$ and $\widetilde{d \hat w}$ described in Proposition~\ref{p:supp0}.

On the other hand, the cosection-localized virtual fundamental class has been shown to agree with the original virtual fundamental class of Behrend-Fantechi.

\begin{proposition}\cite{CL, KO, CLcomp}
The cosection-localized virtual fundamental of \cite{CL} is equal to the Behrend--Fantechi virtual fundamental class up to a sign of $(-1)^{\chi(\cL^{\otimes k})}$.
\end{proposition}
In \cite[Theorem~1.1]{CL}, Chang--J. Li prove that the degree of the virtual classes agree in the case of stable maps to a quintic three-fold with $n=0$.  In \cite[Theorem~3.2]{KO} and \cite[Theorem~1.1]{CLcomp}, Kim--Oh and Chang--M.-L. Li  prove that the virtual classes themselves agree.

\section{Conclusions and further directions}\label{s:conc}

In these notes we have used two-periodic Koszul complexes and the localized top Chern class to give
 a new construction of the virtual fundamental class for Gromov--Witten theory of a projective hypersurface.  This method was first developed by Polishchuk--Vaintrob in \cite{PVold}, where these ideas were employed to give a mathematical definition of the Witten top Chern class on the moduli space of spin curves.
%
 
 In fact the methods described above apply in greater generality than has been treated here.  The general context in which one might use such techniques is that of a \emph{gauged linear sigma model} (GLSM).  This includes both Gromov--Witten theory of a complete intersection as well as FJRW theory of a singularity as special cases.  
 For a detailed description of the mathematical theory of the GLSM, as well as a construction of corresponding enumerative invariants in certain  cases, see \cite{FJR15}.  
%
%
 
Although less familiar to mathematicians than Gromov--Witten theory, one place in which GLSMs arise naturally is through mirror symmetry and wall crossing correspondences with Gromov--Witten theory.  The most well-known example of this is the Landau--Ginzburg/Calabi--Yau correspondence of \cite{ChR, CIR, LPS}, relating Gromov--Witten theory and FJRW theory of the quintic.  An analogous result involving more exotic GLSMs, known as \emph{hybrid models}, appears in \cite{Clader}.
 
 In contrast to Gromov--Witten theory of hypersurfaces, for FJRW theory and general hybrid model GLSMs, a more general construction than the two-periodic Koszul complex is needed to define the corresponding enumerative theory.  In this level of generality, two-periodic complexes are replaced by so-called matrix factorizations.  Such a construction was accomplished first in FJRW theory in \cite{PV}.  It was generalized to hybrid model GLSMs in \cite{CFFGKS}.  

\bibliographystyle{amsplain} 

\bibliography{VC}{}

\end{document}